\documentclass[11pt, leqno]{amsart}

\usepackage{amsmath,amssymb,amsthm, epsfig}

\usepackage{amsmath}
	
\usepackage{amssymb}

\usepackage{color}

\usepackage{ulem}

\usepackage{epsfig}

\usepackage[mathscr]{eucal}

\usepackage[usenames,dvipsnames]{pstricks}
\usepackage{epsfig}
\usepackage{pst-grad} 
\usepackage{pst-plot} 

\usepackage{mathptmx} 
\usepackage[scaled=1]{helvet} 
\usepackage{courier} 
\normalfont
\usepackage[T1]{fontenc}

\title[Free transmission problems]{Free transmission problems}
\author{ Marcelo D. Amaral}
\author{Eduardo V. Teixeira}

\address{PUC-Rio, Departamento de Matem\'atica, Rua Marqu\^es de Sa\~o Vicente 225, G\'avea, Rio de Janeiro - RJ - Brazil, 22451-900}
\email{mamaral@mat.puc-rio.br}
\address{Universidade Federal do Cear\'a, Departamento de Matem\'atica, Campus do Pici - Bloco 914 , Fortaleza, CE - Brazil 60.455-760.}
\email{teixeira@mat.ufc.br} 

\date{}

\def \R {\mathbb{R}}

\def \div {\mathrm{div}}
\def \dist {\mathrm{dist}}

\def \loc {\mathrm{loc}}

\def \suchthat {\ \big | \ }

\def \H {\mathcal{H}^{n-1}}
\def \A {\mathcal{A}}
\def \Leb {\mathscr{L}^n}

\newtheorem{theorem}{Theorem}[section]

\newtheorem{lemma}[theorem]{Lemma}

\newtheorem{proposition}[theorem]{Proposition}

\newtheorem{corollary}[theorem]{Corollary}

\theoremstyle{definition}

\theoremstyle{remark}

\newtheorem{remark}[theorem]{Remark}

\numberwithin{equation}{section}

\newcommand{\intav}[1]{\mathchoice {\mathop{\vrule width 6pt height 3 pt depth  -2.5pt
\kern -8pt \intop}\nolimits_{\kern -6pt#1}} {\mathop{\vrule width
5pt height 3  pt depth -2.6pt \kern -6pt \intop}\nolimits_{#1}}
{\mathop{\vrule width 5pt height 3 pt depth -2.6pt \kern -6pt
\intop}\nolimits_{#1}} {\mathop{\vrule width 5pt height 3 pt depth
-2.6pt \kern -6pt \intop}\nolimits_{#1}}}

\begin{document}

\begin{abstract}  We study transmission problems with free interfaces from one random medium to another. Solutions are required to solve distinct partial differential equations, $\mbox{L}_{+}$ and $\mbox{L}_{-}$, within their positive and negative sets respectively. A corresponding flux balance from one phase to another is  also imposed. We establish existence and $L^{\infty}$ bounds of solutions. We also prove that variational solutions are non-degenerate and develop the regularity theory for solutions of such free boundary problems.

\medskip

\medskip


\end{abstract}

\maketitle

\tableofcontents

\section{Introduction}

The study of transmission problems often refers to the analysis of models involving different media and/or different laws in separate subregions. Such problems appear in several fields of physics, biology, material sciences, industry, etc. It is particularly relevant for mathematical models associated with composite materials. Electromagnetic (or thermodynamic) processes with different dielectric constants (or heat conductivity) are also typical examples of transmission problems. A large literature on this class of problems evolved from the earliest works back in the 1960's. For a detailed historical reference on the development of transmission problems through the decades, we suggest to the readers the recent book \cite{book_B}.  

\medskip

Classical equations related to the mathematical analysis of transmission problems involve discontinuous coefficients. This is due to different properties and distinct features of the media: devices made of different materials, vibrating folded membranes,  multi-structures bodies, anti-plane shear deformation, etc.  Simple mathematical models convert into  the analysis of  second order elliptic equations of the type
\begin{equation}\label{introEq01}
	\nabla \cdot  (a(x) \nabla u) = 0, \quad \text{in } D,
\end{equation}
where
\begin{equation}\label{introEq02}
	a(x) = \left \{ 
		\begin{array}{lll}
		a_0 & \text{in}& D_0 \Subset D\\
		a_1 & \text{in}& D \setminus D_0,
		\end{array}
	\right.
\end{equation}
for some subdomain $D_0$ of $D$ and  $0 < a_0 \not = a_1 < \infty$. For these type of problems, the prior knowledge of $\partial D_0$ as well as its smoothness and geometric properties are essential in their studies. The rigorous mathematical analysis of equations of the form described above, \eqref{introEq01}, \eqref{introEq02}, promoted the development of a number of deep ideas and new mathematical tools,  \cite{LV, LN, BLY, BCN, Dong} just to cite a few recent ones.

\begin{figure}[h!]
\centering
\scalebox{1}  
{
\begin{pspicture}(0,-3.16)(10.267399,3.16)
\definecolor{color37b}{rgb}{0.8,0.8,0.8}
\psbezier[linewidth=0.04](0.6219097,-0.30659932)(0.0,0.4991765)(0.42978415,2.8862627)(2.5815902,3.0131314)(4.733396,3.14)(10.247399,-0.30659932)(6.981265,-1.7232996)(3.7151308,-3.14)(4.001781,-0.07037421)(3.0811167,-0.03171717)(2.160452,0.006939872)(1.2438194,-1.1123751)(0.6219097,-0.30659932)
\pscircle[linewidth=0.04,dimen=outer,fillstyle=solid,fillcolor=color37b](5.5573993,-0.19){1.01}
\usefont{T1}{ptm}{m}{n}
\rput(6.1988544,-1.275){$D_0$}
\usefont{T1}{ptm}{m}{n}
\rput(8.398854,-0.095){$D$}
\usefont{T1}{ptm}{m}{n}
\rput(5.6288543,-0.155){$\mbox{medium } a_0$}
\usefont{T1}{ptm}{m}{n}
\rput(3.0188544,0.995){$\mbox{medium }a_1$}
\end{pspicture} 
}
\vspace{-0.8cm}
\caption{Typical transmission problem}
\end{figure}
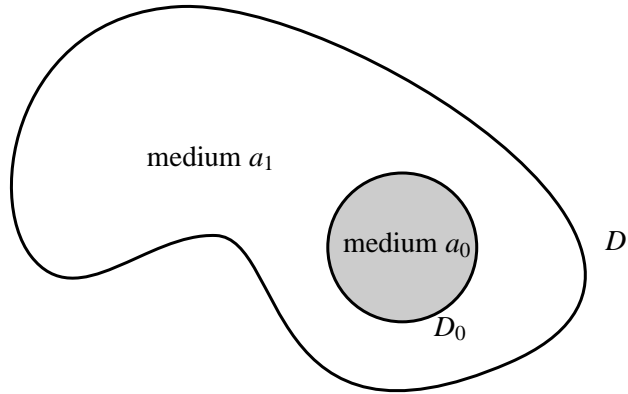

Let us now examine another simple physical situation, for which the transmission problem involved is rather more delicate from the mathematical view point. Consider the system of equations modeling an ice that melts submerged in a heated inhomogeneous medium. Let us focus on the stationary situation -- a snapshot of the phenomenon. When the temperature $T$ is negative, the heat conduction process is governed by the diffusion operator associated to the ice. Thus, after normalization, we can assume that 
\begin{equation}\label{introEq03}
	\Delta T = 0, \ \text{ is satisfied inside the ice}.
\end{equation}
For positive temperatures, the heat conduction process is now ruled by the diffusion operator associated to the exterior medium. Thus,  
\begin{equation}\label{introEq04}
	\nabla \cdot (\bar{a} (x) \nabla T) = 0, \text{ in the exterior fluid,}
\end{equation}
for some $\bar{a}(x)$ positive definite matrix different from the identity matrix.  From the classical laws of thermodynamics,  an extra energy (latent heat) is required for the change of state of the matter. This translates into a prescribed heat flux balance along the phase transition $\{T=0\}$. One finds out the equation 
\begin{equation}\label{introEq05}
	\partial T^{-} - \partial_{\bar{a}} T^{+} = c_0, \ \text{along } \{T=0\},
\end{equation}
for some "latent heat for melting" constant $c_0  \not = 0$.

\begin{figure}[h!]
\centering
\scalebox{1} 
{
\begin{pspicture}(0,-3.59)(9.819823,3.59)
\definecolor{color20b}{rgb}{0.8,0.8,0.8}
\psbezier[linewidth=0.04,fillstyle=solid,fillcolor=color20b](1.7914643,2.6692548)(2.8314817,3.57)(7.3501773,3.501589)(8.426516,2.657909)(9.502853,1.814229)(9.799823,-2.0989447)(8.581462,-2.7389061)(7.3631015,-3.3788674)(2.6329556,-3.57)(1.3164778,-2.616103)(0.0,-1.6622058)(0.7514471,1.7685097)(1.7914643,2.6692548)
\psbezier[linewidth=0.04,fillstyle=solid](2.4640794,1.8852576)(2.9068406,2.43)(3.084843,1.9615359)(4.4140916,1.9760482)(5.74334,1.9905604)(6.016862,2.3960457)(6.529049,1.9674678)(7.0412354,1.5388899)(6.6844625,1.1848748)(6.6566052,0.069449894)(6.628748,-1.045975)(7.199823,-1.4854436)(6.576579,-1.8523877)(5.9533353,-2.2193315)(5.7791,-1.8501184)(4.3862343,-1.9019986)(2.9933684,-1.9538789)(2.633442,-2.1743698)(2.2966325,-1.8189514)(1.959823,-1.4635329)(2.4640794,-0.9551846)(2.5058653,0.1861804)(2.5476513,1.3275454)(2.0213182,1.3405153)(2.4640794,1.8852576)
\usefont{T1}{ptm}{m}{n}
\rput(4.6212783,0.335){$\mbox{Ice}$}
\usefont{T1}{ptm}{m}{n}
\rput(2.9812782,-1.545){$T<0$}
\usefont{T1}{ptm}{m}{n}
\rput(4.501278,2.575){$\mbox{Exterior fluid}$}
\usefont{T1}{ptm}{m}{n}
\rput(7.901278,2.395){$T>0$}
\usefont{T1}{ptm}{m}{n}
\rput(4.491278,1.415){$\Delta T = 0$}
\usefont{T1}{ptm}{m}{n}
\rput(5.6412783,-2.485){$\nabla \cdot (\bar{a}(x)\nabla T) = 0$}
\psline[linewidth=0.04cm,arrowsize=0.05291667cm 2.0,arrowlength=1.4,arrowinset=0.4]{->}(6.639823,0.23)(7.739823,1.33)
\usefont{T1}{ptm}{m}{n}
\rput(8.211278,1.535){$\partial T^{-}$}
\psline[linewidth=0.04cm,arrowsize=0.05291667cm 2.0,arrowlength=1.4,arrowinset=0.4]{->}(6.639823,0.23)(5.419823,-0.37)
\usefont{T1}{ptm}{m}{n}
\rput(5.401278,-0.665){$\partial_{\bar{a}} T^{+}$}
\end{pspicture} 
}\caption{Ice melting within an exterior medium}
\end{figure}
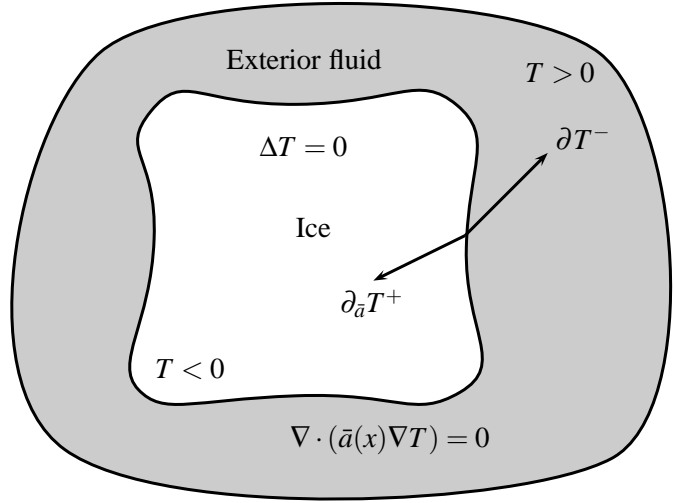

Albeit similar in nature, two crucial differences ought to be highlighted from the typical transmission problem and the ice melting situation described above. First notice that in the latter,  the transmission boundary is {\it a priori} unknown, and in fact depends on the solution itself. No smoothness properties on $\{T=0\}$ are {\it a priori} granted  and, in general, it may fail to be a Lipschitz surface. This is in contrast to the fact that, say, in composite body problems,  the discontinuity transmission occurs along a fixed, smooth prescribed boundary, $\partial D_0$. The second different feature of the ice melting problem rests upon the flux balance, entitled by equation \eqref{introEq05}. Unifying  the equations involved in the system, one formally ends up with a nonhomogeneous elliptic equation with a measure datum:
\begin{equation}\label{introEq06}
	\nabla \cdot  ( \mathcal{A}(x) \nabla T ) = \mu,  
\end{equation}
where $\mathcal{A}(x) = \text{Id}$ for $T<0$, $\mathcal{A}(x) = \bar{a}(x) $ for $T>0$, and $\mu$ is a nonzero measure supported along the free interface $\{ T = 0 \}$, in particular $\mu$ is not absolutely continuous with respect to the $\Leb$-Lebesgue measure. Should $\{T=0\}$ be an  $(n-1)$ smooth surface, then $\mu = c_0 \lfloor \{T=0\}$ in the sense of measures. 

\medskip 

The main goal of present work is to provide a mathematical treatment to transmission problems involving free boundaries of discontinuity, as the ice melting example illustrated above. For that, let us start describing the mathematical set-up we shall work on in this paper. Let $\Omega$ be a bounded open subset of $\R^n, ~n \ge 3$ with Lipschitz
boundary and $A_+$, $A_-\in L^\infty(\Omega,{\rm Sym}(n\times n))$ be functions of symmetric matrices satisfying the ellipticity condition
 \begin{equation}\label{elliptic}
	\lambda|\xi|^2\le\langle A_\pm(x)\,\xi,\xi\rangle \le \Lambda|\xi|^2,
\end{equation}
 for a.e.  $x\in\Omega$ and all $\xi\in\R^n$, where  $0<\lambda \le \Lambda$ are fixed constants. Let also  $\gamma \colon \Omega \to \mathbb{R}$ be a continuous, strictly positive  function and  $0< \lambda_-<\lambda_+<\infty$ be different constants.  Finally select sources functions $ f_{+}, f_{-} \in L^{q}(\Omega)$, with $ q > n/2$. 
\par

\medskip 

Given a fixed boundary datum  $\phi \in H^1(\Omega)\cap L^\infty(\Omega)$, we consider the  energy functional 
${\mathcal{F}} = {\mathcal{F}}_{A_\pm, \lambda_\pm,f_\pm} \colon  H^1_{\phi}(\Omega)\rightarrow [0,+\infty)$ 
defined by 
\begin{equation} \label{F}
	\begin{array}{lll}
	 	{\mathcal{F}}_{A_\pm, \lambda_\pm,f_\pm}(v)&:= & \displaystyle  \int\limits_{\Omega \cap\{ v>0\}} \left \{\dfrac{1}{2} |\nabla v(x)|^2_{A_+} + \lambda_{+}(x) + f_{+}v \right \}  \,dx \\
		&+&
		\displaystyle  \int\limits_{\Omega \cap\{ v\le0\}} \left \{\dfrac{1}{2}|\nabla v(x)|^2_{A_-}  + \lambda_{-}(x) + f_{-}v \right \}\,dx,
	\end{array}
\end{equation}
where we use the notation
$$|\nabla v(x)|^2_{A_\pm}:=
\langle A_\pm (x)\,\nabla v(x),\nabla v(x)\rangle, 
$$
and 
$$
	\lambda_+(x):= \gamma(x)\lambda_+ \quad \mbox{and} \quad \lambda_-(x):= \gamma(x)\lambda_{-},
$$  
As usual $H^1_{\phi}(\Omega)$ denotes  the functional affine manifold
$$ 
	H^1_{\phi}(\Omega):=\{v\in H^1(\Omega) \suchthat  v-\phi\in H^1_0(\Omega)\}.
$$

\medskip 

Some notation simplifications shall be used throughout the paper. Often we shall  omit the subscripts of the energy functional.  We will simply write ${\mathcal{F}}(v)$ when the choices of $A_\pm,\lambda_\pm,f_\pm$ are understood in the context. It will be convenient to write ${\mathcal{F}_A}(v)$ when  $A_+ = A_- = A$. Throughout the text , it will also be useful to write the functional \eqref{F} as follows  
\begin{equation} \label{F1} \tag{P}
\mathcal{F}(u) = \int\limits_{\Omega}  \left \{ \dfrac{1}{2}\langle A(x,u) \nabla u, \nabla u \rangle + F_0(u) + fu \right \} dx
\end{equation}
where 
\begin{eqnarray}
	 A(x,u)&:=&  A_+\chi_{\{u>0\}} + A_-\chi_{\{u\le0\}} \label{A(x,u)} \\
	F_0(u) &=& \lambda_{+} \chi_{\{u>0\}} + \lambda_{-} \chi_{\{u \le 0\}} \label{F(u)} \\
	f&=&f_{+}\chi_{\{u>0\}} + f_{-}\chi_{\{u\le 0\}}. \label{f}
\end{eqnarray}

\medskip 

Extrema functions of the energy $\mathcal{F}$ are related to free transmission problems earlier illustrated. Indeed, it is expected that local minima of $\mathcal{F}$ satisfy
$$ 
	 \nabla \cdot (A_+(x) \nabla u(x)) = f_{+}(x) \  \mbox{in} \ \{u \ > 0\}
$$
and analogously, one should verify that
$$ 
	 \nabla \cdot  (A_-(x) \nabla u(x)) = f_{-}(x) \  \mbox{in} \ \{u \le 0\}^\circ.
$$
An appropriate free boundary condition will also emerge from small domain variations near the  free discontinuity surface $\partial \{u> 0 \}$, see Section \ref{sct FBC}. Hence, the mathematical treatment of free transmission problems leads naturally to the study of local minima of the energy functional $\mathcal{F}$ --  this is the goal of this present work. 

\medskip 

It is as well alluding to observe that the energy functional $ {\mathcal{F}}_{A_\pm, \lambda_\pm,f_\pm}$ described in \eqref{F} can also be viewed as a generalization of the Alt-Caffarelli-Friedman theory, developed henceforth the epic-marking paper \cite{ACF}. For that, one simply take $A_{+} = A_{-} = \text{Id}$, $f_{+} = f_{-} = 0$. Sometimes it will be convenient to explore such a perspective, confronting the results obtained in \cite{ACF} with the ones proven herein for the general free transmission problem. 
\medskip

Before continuing, let us declare that any constant or entity that depends only upon dimension, ellipticity constants of the media entitled in \eqref{elliptic}, $\gamma$, $\lambda_{+}$, $\lambda_{-}$, $\|f_{+}\|_q$, $\|f_{-}\|_q$, $\Omega$ and $\phi$, shall be called universal. Alternatively, we will say that a constant depends only upon the data of the problem.

\medskip

Discontinuities of the media along the free interface impel several new subtleness and technical difficulties in the mathematical treatment of this type of problems.  We highlight that existence of a minimizer {\it does not} follow from classical methods in the Calculus of Variations. The main difficulty is the lack of convexity of the functional \eqref{F}. We exemplify that at the beginning of Section \ref{ELIB}. 

\medskip 

Despite of the lack of convexity, we shall prove the functional \eqref{F} has a minimizer, Theorem \ref{ET}. We will also obtain an $L^\infty$ estimate of a minimum, that depends only upon the data of the problem. These two results will be delivered in Section \ref{ELIB}. 

\medskip

In Section \ref{HE}, we refine our estimates as to prove that any local minima of the functional $\mathcal{F}$ has a universal modulus of continuity, Theorem \ref{reg}. This will be a key tool in the study of further qualitative properties of minima. In particular, Theorem \ref{reg} provides compactness of minima, which plays a decisive role in the perturbative approach we shall establish henceforth.
\medskip

In the sequel, we start developing a geometric regularity theory for local minima. We show that such extrema grow linearly away from the transmission interface, Theorem \ref{CL}. As a consequence, we establish strong nondegeneracy properties of minima along the phase transition, Theorem \ref{TNDF}. These results are carried out in Section \ref{GN}.

\medskip

The next step in the program is to analyze the optimal asymptotic regularity theory for minima of  when $A_{+} = A_{-} = A(x)$ is a mere continuous function of symmetric elliptic matrices. We recall that in such a scenario, even solutions to the homogeneous equation
\begin{equation} \label{hom_int_FHR}
    \div (A(x) \nabla u) = 0
\end{equation}
may fail to be Lipschitz continuous, see for instance \cite{JMV}. Notwithstanding, we shall deliver in Section \ref{FHR} an asymptotic optimal $C^{0,1^{-}}$ regularity estimate for minima of
$\mathcal{F}_{A, f_{+}, f_{-}}$ for $f_{-}, f_{+} \in L^q$, $q\ge n$. As to further enlighten the subtleness of the asymptotic estimate we shall deliver in Section \ref{FHR}, we comment that if $A$ had $\beta$-H\"older continuous coefficients, $0< \beta< 1$, then for the homogeneous equation \eqref{hom_int_FHR}  classical Schauder estimates provide $C^{1,\beta}$ smoothness of solutions. In parallel, under H\"older continuity of the coefficients, it is possible to establish a monotonicity formula for  $\mathcal{F}_{A,0,0}$, see \cite[Lemma 1]{C3}, and hence minima are Lipschitz continuous.

\medskip

Before presenting the ultimate asymptotic Lipschitz regularity estimate we shall prove in Section \ref{sct small jumps}, we invite the readers to revisit the ice-melting problem previously illustrated at the beginning of this Introduction. One should notice that while the ice in melting inside a different medium, a mixing process is taking place near the free boundary. The water coming from the  melted ice associates with the exterior medium, turning its heat properties closer and closer to the heat properties of the pure water. In mathematical terms, when one looks at the evolutionary problem, the exterior medium $A_{+}$ depends on the time parameter $t$ and we verifies that
\begin{equation} \label{mot intro}
    \lim\limits_{t \to \infty} \|A_{-} - A_{+}(t) \| = 0,
\end{equation}
near the moving free interface $\partial \{ T(x,t) > 0 \}$. That is, after some time, the jumping discontinuity from one medium to another should be no more than $\epsilon$, a given choice of closeness, $\epsilon>0$.

\medskip

With this illustrative physical intuition in mind, we shall prove in our final Theorem \ref{HoR} that, under $\epsilon$-smallness of the jump discontinuity, solutions are locally of class $C^{0,\alpha(\epsilon)}$ and
$$
    \lim\limits_{\epsilon \to 0}  \alpha(\epsilon) = 1.
$$
That is, solutions to the free transmission problem are asymptotically Lipschitz continuous as $t \to \infty$. It is worth commenting that even in the case that $A_{+}$ and $A_{-}$ are (different) constant elliptic matrices, one cannot establish monotonicity formula for the homogeneous $\mathcal{F}_{A_{+}, A_{-}, 0,0}$ functional (cf. \cite{CS}). As a consequence, even in the simplest heterogeneous scenario, Lipschitz estimates may not be valid for local minima. In that perspective, the asymptotic estimate proven herein  is of optimal nature.

\medskip

{\it Acknowledgements.} The authors would like to thank Enrico Valdinoci and Luis Caffarelli for the their caring interest on this program and for fruitful discussions. The authors researches have been partially funded by CNPq-Brazil.

\section{Existence and $L^\infty$ bounds} \label{ELIB}

In this section we establish existence of a minimum to the functional \eqref{F}  as well  as universal $L^\infty$ bounds for such an extremum. As previously mentioned, the functional \eqref{F} fails, in general, to be convex. To exemplify this we invite the readers to analyze  the following one-dimensional example. Let $\Omega=(0,10), ~A_+ \equiv 20$ and $A_- \equiv 1$, $\lambda_{+} = \lambda_{-} = f_{+} = f_{-} = 0$. Define the functions $u$ and $v$ by the graphs below:
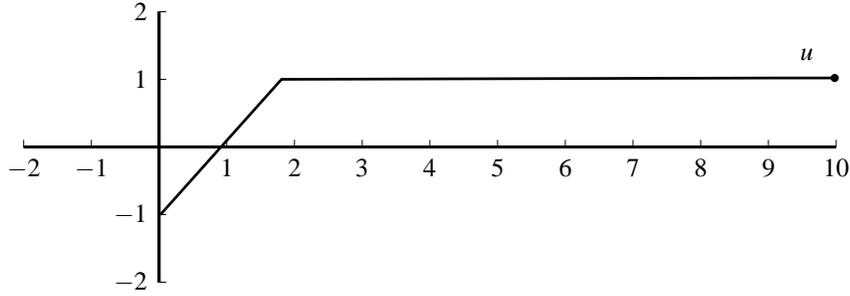
\begin{figure}[h!]
\centering
\scalebox{0.9} 
{
\begin{pspicture}(0,-2.0)(12.06,2.0)
\rput(2.0,0.0){\psaxes[linewidth=0.04,ticksize=0.10583333cm](0,0)(-2,-2)(10,2)}
\psline[linewidth=0.04](2.02,-1.0)(3.8083212,1.0)(11.144088,1.02)(12.02,1.02)
\usefont{T1}{ptm}{m}{n}
\rput(11.571455,1.365){$u$}
\psdots[dotsize=0.12](11.98,1.02)
\end{pspicture} 
}
\caption{Graph of competing function $u$.}
\end{figure}


\begin{figure}[h!]
\centering
\scalebox{0.9} 
{
\begin{pspicture}(0,-2.0)(12.12,2.0)
\rput(2.0,0.0){\psaxes[linewidth=0.04,ticksize=0.10583333cm](0,0)(-2,-2)(10,2)}
\usefont{T1}{ptm}{m}{n}
\rput(11.511456,1.265){$v$}
\psline[linewidth=0.04](2.0,-1.0)(3.02,-0.02)(3.98,-1.02)(5.02,0.0)(6.06,-1.06)(7.02,-0.04)(8.0,-0.98)(9.0,-0.02)(10.02,-1.0)(12.02,1.0)(12.02,1.0)
\psdots[dotsize=0.12](12.04,1.0)
\end{pspicture} 
}

\caption{Graph of competing function $v$.}
\end{figure}
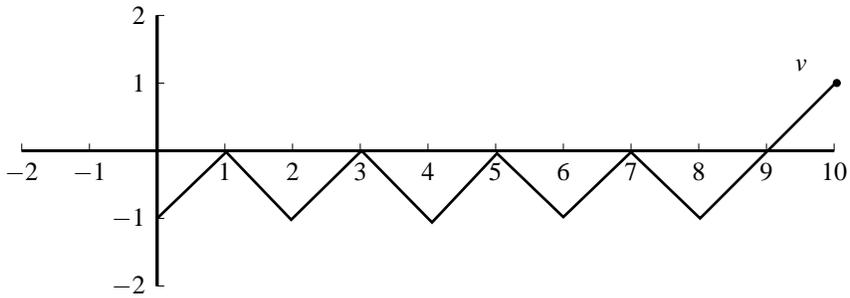

Notice that both functions take the same value on the boundary of $\Omega$. The corresponding functional  
$$
	\mathcal{G}(\psi) = 20\int\limits_{ \chi_{\{\psi>0\}\cap \Omega}}(\psi')^2 dx+ \int\limits_{\chi_{\{ \psi \le0\}\cap \Omega}}(\psi')^2dx.
$$ 
is not convex. Indeed, one easily compute:
$$
	\mathcal{G}\Bigl(\dfrac{u+v}{2}\Bigr) = 1+8\cdot \dfrac{1}{4}\cdot20=41 \ge \dfrac{1}{2}21 + \dfrac{1}{2}29 =\dfrac{1}{2}\mathcal{G}(u) +\dfrac{1}{2} \mathcal{G}(v).
$$ 
It is clear from the construction above that, if desired, one can wig the functions as to fail the concavity inequality as well.  

Combining methods from the Calculus of Variations with theoretical measure estimates, we can, nonetheless, show the existence of a local minima, see for instance \cite{T-1, A} for similar approaches. 

\begin{theorem}[Existence theorem] \label{ET}
There exists a minimum $u_0$ to the energy functional \eqref{F}. 
\end{theorem}
\begin{proof}
For this proof, it will be convenient to write the functional $\mathcal{F}$ as indicated in \eqref{F1}. From ellipticity of the matrices $A_{+}$ and $A_{-}$ we readily obtain the lower bound
\begin{equation} \label{EM1}
	\mathcal{F}(u) \ge \dfrac{\lambda}{2} \int\limits_{\Omega} (| \nabla u|^{2} + 2fu) dx.
\end{equation} 
Clearly, $q> n/2 > 2n/(n+2)$, thus by H\"older inequality we can estimate
\begin{equation} \label{EM2}
	\left |  \int\limits_{\Omega} fu dx \right |   \le  C_1(n,\Omega) \|f\|_{L^{q}(\Omega)} \| u\|_{L^{2n/(n-2)}(\Omega)}.
\end{equation}
Hence, combining  \eqref{EM1} and \eqref{EM2}, together with Sobolev embedding and Poincar\'e inequality, yields
\begin{equation} \label{EM22}
\mathcal{F}(u) \ge \dfrac{\lambda}{2} \int\limits_{\Omega} | \nabla u|^{2}dx - C(n,\lambda,\|f\|_{L^q(\Omega)}, \phi, \Omega) \left ( \int\limits_{\Omega} | \nabla u|^{2}dx \right )^{1/2} - C(\phi).
\end{equation}
We have verified that  the functional $\mathcal{F}$ is coercive along  $H^{1}_{\phi}(\Omega)$ and thus
\begin{equation}\label{EM29}
	\inf \left  \{\mathcal{F}(u) \suchthat  u \in H^{1}_{\phi}(\Omega) \right \} > - \infty.
\end{equation}	
  
In the sequel, let $\{u_m\}_{m\ge 1} \subset H_\phi^{1}(\Omega)$ be a minimizing sequence for $\mathcal{F}$. Passing to a subsequence, if necessary, we can assume that $u_m \rightarrow u_0$ a.e in $\Omega$ and $\nabla u_m \rightharpoonup \nabla u_0$ in $ L^2(\Omega)$. From Egoroff's theorem, for any $\varepsilon >0 $ there is an open subset $\Omega_\varepsilon \subset \Omega$, with  $|\Omega \backslash \Omega_\varepsilon|< \varepsilon$, for which we can assure $u_m \rightarrow u_0$ uniformly in $\Omega_\varepsilon$. Now, given $\delta > 0$, 
we estimate
\begin{eqnarray*}
\int\limits_{ \Omega_\varepsilon \cap \{u_0> \delta\}} \dfrac{1}{2}\langle A(x,u_0) \nabla u_0, \nabla u_0 \rangle dx &=& \hspace{-0.3cm}
\int\limits_{ \Omega_\varepsilon \cap \{u_0> \delta\}} \dfrac{1}{2}\langle A_+ \nabla u_0, \nabla u_0 \rangle dx \\
&\le& \liminf_{m \rightarrow \infty}  \hspace{-0.3cm} \int\limits_{ \Omega_\varepsilon \cap \{u_0> \delta\}} \dfrac{1}{2}\langle A_+ \nabla u_m, \nabla u_m \rangle dx \\
&\le& \liminf_{m\rightarrow \infty}  \hspace{-0.3cm}  \int\limits_{ \Omega_\varepsilon \cap \{u_m> \delta / 2\}} \dfrac{1}{2}\langle A_+ \nabla u_m, \nabla u_m \rangle dx \\
&\le& \liminf_{m \rightarrow \infty}  \hspace{-0.3cm}  \int\limits_{ \Omega_\varepsilon \cap \{u_m > 0\}} \dfrac{1}{2}\langle A_+ \nabla u_m, \nabla u_m \rangle dx \\ 
&\le& \liminf_{m \rightarrow \infty}  \hspace{-0.3cm} \int\limits_{\Omega \cap  \{u_m > 0\}} \dfrac{1}{2}\langle A_+ \nabla u_m, \nabla u_m \rangle dx \\ 
&=& \liminf_{m \rightarrow \infty}  \hspace{-0.3cm}  \int\limits_{\Omega \cap  \{u_m > 0\}} \dfrac{1}{2}\langle A(x,u_m) \nabla u_m, \nabla u_m \rangle dx. \\
\end{eqnarray*}
Letting $\delta \to 0$  in the above chain of inequalities yields
\begin{equation} \label{epsilon}
\int\limits_{ \Omega_\varepsilon \cap \{u_0> 0\}}   \hspace{-0.3cm}  \langle A(x,u_0) \nabla u_0, \nabla u_0 \rangle dx \le \liminf_{m \rightarrow \infty} \int\limits_{\Omega \cap \{u_m > 0\}}  \hspace{-0.3cm} \langle A(x,u_m) \nabla u_m, \nabla u_m \rangle dx. 
\end{equation}
On the other hand, by ellipticity and $L^2$ bounds on $\nabla u_0$, we have
\begin{equation}\label{est remain}
	 \int\limits_{ (\Omega \backslash \Omega_\varepsilon) \cap \{ u_0 > 0\}} \langle A(x,u_0) \nabla u_0, \nabla u_0 \rangle dx = \text{O}(\varepsilon).
\end{equation}
Hence, combining \eqref{epsilon} and \eqref{est remain} and, afterwards letting $\varepsilon \to 0$,  gives
\begin{equation} \label{EM3}
\int\limits_{\Omega \cap \{u_0> 0\}} \langle A(x,u_0) \nabla u_0, \nabla u_0 \rangle dx \le \liminf_{m \rightarrow \infty} \int\limits_{ \Omega \cap \{u_m > 0\}} \langle A(x,u_m) \nabla u_m, \nabla u_m \rangle dx.
\end{equation}
Reasoning analogously, we obtain  
\begin{equation} \label{EM4}
\int\limits_{\Omega \cap \{u_0 \le 0\}} \langle A \nabla u_0, \nabla u_0 \rangle dx \le \liminf_{m \rightarrow \infty} \int\limits_{\Omega \cap \{u_m < 0\}} \langle A \nabla u_m, \nabla u_m \rangle dx.
\end{equation}
Lower weak semicontinuity of the other terms of $\mathcal{F}$ is standard. That is, that
\begin{equation} \label{EM5}
\int\limits_{\Omega}    fu_0 dx \le \liminf_{m \rightarrow \infty} \int\limits_{\Omega}   fu_m dx
\end{equation}
and
\begin{equation} \label{EM6}
\int\limits_{\Omega}F_0(u_0) dx \le \liminf_{m \rightarrow \infty} \int\limits_{\Omega} F_0(u_m) dx,
\end{equation}
are classical facts. From (\ref{EM3}), (\ref{EM4}), (\ref{EM5}) and (\ref{EM6}), it follows that the limiting function $u_0$ is a minimum of $\mathcal{F}$ and Theorem \ref{ET} is concluded.
\end{proof}

In the sequel we establish a universal control on the $L^\infty$ norm of the minimum granted by Theorem \ref{ET}. Such an estimate will play an important role in higher order estimates we shall prove later in the program.

\begin{theorem}[$L^\infty$ bounds] \label{LIB} There exists a  constant $C>0$ depending only upon dimension,  ellipticity constants $\lambda$ and $\Lambda$, and the numbers $\lambda_{+}$, $\lambda_{-}$, $\|\lambda(x)\|_\infty$, $\| \phi\|_{L^{\infty}(\partial \Omega)}$ and  $\|f\|_{L^{q}(\Omega)}$ such that
\begin{equation}
\| u_0 \|_{L^{\infty}(\Omega)} \le C,
\end{equation}
holds true for any local minimum $u_0$ of the functional $\mathcal{F}$.
\end{theorem}
\begin{proof}
Indeed, let $u$ be a minimum of the energy functional entitled in \eqref{F1} and $j_0$  the smallest natural number above $\| \phi\|_{L^{\infty}(\partial \Omega)}$. For each $j \ge j_0$ we define the truncated function $u_j:\Omega \rightarrow \mathbb{R}$ by
$$
u_j = \left\{
\begin{array}{rcl}
j \cdot \mbox{sign}(u),& \mbox{if} & |u| > j\\
u, & \mbox{if} & |u| \le j\\
\end{array}
\right.
$$
where sign$(u)(x) = 1$ if $u(x) > 0, 0$ if $u(x)=0$ and sign$(u)(x) = -1$ if $u(x) < 0$. Denote by
$$
	\A_j : = \left \{x \in \Omega \suchthat  |u(x) | > j \right \}.
$$ 
It follows from the very definition of the above set that for each $j > j_0$,  one verifies
\begin{equation}
u_j = u \quad \mbox{in} \quad \A^{c}_{j} \quad \mbox{and} \quad u_j = j\cdot\mbox{sing}(u) \quad \mbox{in} \quad \A_j.
\end{equation}
Thus, by minimality of $u$, we can estimate
\begin{eqnarray*}
\int\limits_{\A_j} \langle A \nabla u, \nabla u \rangle - \langle A \nabla u_j, \nabla u_j \rangle dx
& \le& \int\limits_{\A_j} f(u_j -u)dx \\
&= & \int\limits_{\A_j\cap \{u>0\}} f(j -u)dx + \int\limits_{\A_j\cap \{u\le0 \}} f(-j -u)dx \\
&\le & 2 \int\limits_{\A_j} |f|(|u| -j) dx.
\end{eqnarray*}
From the range of truncation we consider, it follows that $(|u|-j)^{+} \in W^{1,2}_{0}(\Omega)$. Thus, applying  H\"oder inequality, Gagliardo-Nirenberg inequality and Cauchy's inequality, we obtain
\begin{eqnarray*}
\int_{\A_j} |f|(|u| -j)^{+} dx & \le & \|f\|_{L^{\frac{2^*}{2^*-1}}(\A_j)} \| (|u| - j)^{+} \|_{L^{2^*}(\A_j)} \nonumber \\
 &\le & C\|f\|_{L^{q}(\A_j)} |\A_j|^{(1 -1/2^{*} - 1/q)} \| \nabla u\|_{L^{2}(\A_j)}\nonumber \\
 &\le & \dfrac{C^2}{4 \varepsilon}\|f\|^2_{L^{q}(\A_j)} |A_j|^{2(1 -1/2^{*} - 1/q)}+ \varepsilon \| \nabla u\|^2_{L^{2}(\A_j)} \nonumber\\
 &\le & C| \A_j|^{(1 - 2/n + \delta)} + \varepsilon \| \nabla u \|^{2}_{L^{2}(\A_j)}.
\end{eqnarray*}
In the last inequality,
$$
 \delta = \dfrac{2(2q-n)}{nq}
$$ 
is a positive number since $q>n/2$. Taking, in the sequel, $\varepsilon = \lambda/2$  we obtain, from ellipticity, that
\begin{equation} \label{L11}
\int_{\A_j} | \nabla u |^{2} dx \le C| \A_j|^{1 - (2/n) + \delta}.
\end{equation}
 Now, we employ Poincar\'e inequality followed by  H\"older inequality to establish the control
$$ 
	\|u\|_{L^{1}(\A_{j_0})} \le | \A_{j_0}|^{1/2} \|u\|_{L^{2}(\A_{j_0})} \le C.
$$
Boundedness of $u$ now follows from a general machinery, see for instance \cite[Chapter 2, Lemma 5.3]{Lad}.
\end{proof}

\begin{remark} \label{rmk H1}We conclude this Section commenting that universal $L^{\infty}$ bounds for minimizers of \eqref{F} yield universal control on the $H^{1}$-norm of extrema. Indeed, one simply verifies
\begin{eqnarray*}
\lambda \int_{\Omega} | \nabla u|^{2} dx &\le & \int_{\Omega} \langle A \nabla u, \nabla u\rangle  dx \\
&\le &  \mathcal{F}(\phi) - \int_{\Omega} F_0(u) dx + \int_{\Omega} |f||u|dx\\
 &\le&\mathcal{F}( \phi) + C(n,\lambda_+,\lambda_-,\|\lambda\|_{L^\infty(\overline{\Omega})},\Omega,\|u\|_{L^{\infty}(\Omega)},\|f\|_{L^{q}(\Omega)})\\
 &\le& C.
\end{eqnarray*}
Such a remark will often be used throughout the next Sections.
\end{remark}
\section{Universal H\"older estimates} \label{HE}

In this present Section we shall prove that local minima to \eqref{F} are universally H\"older continuous. We recall the {\it magnum opus} of elliptic regularity theory:   the De Giorgi--Nash--Moser Theorem. It asserts that a weak solution to a divergence form, uniformly elliptic equation
\begin{equation} \label{alphaA}
	\mbox{div}(A(x)\nabla u) = 0
\end{equation}
is locally $\alpha$-H\"older continuous, for a universal exponent $0 < \alpha = \alpha_{ \lambda,\Lambda } < 1$ that depends only upon dimension and ellipticity constants, $0 < \lambda \le \Lambda < \infty$.  As we shall see later in the program, minimizers to the functional $\mathcal{F}$ do satisfy an elliptic equation in the sense of distributions. However the RHS of the equation involves a singular measure supported along the free boundary. Thus, the Euler-Lagrange equation related to $\mathcal{F}$ is of out of the scope of that classical regularity theory and a new approach is required.

Herein we shall use the classical average notation
$$ 
	(f)_{r,x_0} :=  \intav{B_r(x_0)} f dx = \dfrac{1}{|B_r(x_0)|} \int_{B_r(x_0)} f dx.
$$
When the center point is understood, we shall simply write $(f)_r$. Our approach starts off by a standard gradient estimate for elliptic equations in the divergence form, which can be proven by pure energy considerations, see for instance \cite{G_Book}, and we shall omit here.
\begin{lemma} \label{LL}
Let $v$ be a solution of 
$ \mbox{div}(A(x)\nabla v) = 0 $ in a ball $B_R(x_0)\subset \mathbb{R}^n$. Then, for a constant $0< \alpha = \alpha({\lambda,\Lambda})<1$  and for any $0<r \le R$ there holds
$$  
	\int\limits_{B_r(x_0)} | \nabla v - (\nabla v)_r|^{2}dx \le C(n,\lambda, \Lambda)\Bigl( \dfrac{r}{R}\Bigr)^{n - 2 + 2\alpha} \int\limits_{B_R(x_0)}| \nabla v(x) -(\nabla v)_{R}|^{2}dx. 
$$ 
\end{lemma}
Owing to Lemma \ref{LL} it is standard to obtain the following $A$-replacement energy estimate, whose proof shall also be omitted.

\begin{lemma} \label{DL}
Let $u \in H^{1}(B_R)$ and $h \in H^{1}(B_R)$ be a solution of $\mbox{div}(A \nabla h) = 0$ in the distributional sense in $B_R(x_0)$. Then, there exist constant $0< \alpha = \alpha({\lambda,\Lambda})<1$ and $C = C(n,\lambda,\Lambda)>0$ such that for any $0<r \le R$
\begin{eqnarray}
\int\limits_{B_r(x_0)} | \nabla u(x) - (\nabla u)_{r}|^2dx &\le& C \Big( \dfrac{r}{R}\Big)^{n-2 +2\alpha}  \int\limits_{B_R(x_0)} |\nabla u(x) - (\nabla u)_R|^2 dx  \nonumber \\
&+&  C\int\limits_{B_R(x_0)} | \nabla u(x) - \nabla h(x)|^2dx.
\end{eqnarray}

\end{lemma}

\medskip

The final ingredient we need is a simple real analysis result concerning estimates of non-decreasing functions. The proof is elementary and shall be omitted. 
 
\begin{lemma}
\label{elem.ineq}
Let $\phi(s)$ be a non-negative and non-decreasing function. Suppose that
\begin{eqnarray}
\phi \left( r \right) \leq C_{1} \left[ \left( \dfrac{r}{R}\right)^{\alpha} + \mu \right] \phi \left( R \right) + C_{2}R^{\beta}
\end{eqnarray}
for all $r \leq R \leq R_{0}$, with $C_{1}, \alpha, \beta$ positive constants and $C_{2}, \mu$ non-negative constants, $\beta < \alpha$. Then, for any $\sigma < \beta$, there exists a constant $\mu_{0}=\mu_{0}\left( C_{1}, \alpha, \beta, \sigma \right)$ such that if $\mu < \mu_{0}$, then for all $r \leq R \leq R_{0}$ we have
\begin{eqnarray}
\label{a 1.1}
\phi \left( r \right) \leq C_{3} \left( \dfrac{r}{R}\right)^{\sigma} \phi \left( R \right) + C_{2}R^{\sigma} \end{eqnarray}
where $C_{3}=C_{3}\left( C_{1}, \sigma - \beta   \right)$ is a positive constant. In turn,
\begin{eqnarray}
\label{a 1.1.1}
    \phi(r) \le C_{4} r^{\sigma},
\end{eqnarray}
where $C_{4}=C_{4}(C_{2}, C_{3}, R_{0}, \phi, \sigma)$ is a positive constant.
\end{lemma}

We now state and prove the main regularity result from this Section. 

\begin{theorem}[Universal H\"older regularity] \label{reg}
Let $u$ be a minimum of the problem \eqref{F1}, with $f \in L^{q}(\Omega), q> n/2$. Then, $u \in C^{\tau}_{\loc}(\Omega)$, for some $0< \tau= \tau(\lambda, \Lambda, q) < 1$. Furthermore, for any $\Omega' \Subset \Omega$, there exists a constant $K$ depending only on  $\Omega'$ and the data of the problem such that
$$
	\|u\|_{C^\tau(\Omega')} \le K.
$$

\end{theorem}
\begin{proof}
Given a minimum $u$  of the functional (\ref{F1}), let us denote $A= A(x,u)$. Fix a point $x_0 \in \Omega$ and $R>0$ such that $R < \mbox{dist}(x_0, \partial \Omega)$. Hereafter we denote $B_R = B_R(x_0)$. 
Let $h$ be the $A$-harmonic function in $B_R$ that agrees with $u$ on the boundary, i.e.,  
\begin{equation} \label{TR1}
\mbox{div}(A \nabla h) = 0 \quad \mbox{in   } B_R, \quad  \mbox{   and   }  \quad h-u \in H^{1}_{0}(B_R).
\end{equation}
One verifies the following integral identity
\begin{equation} \label{TR2}
  \int_{B_R} \dfrac{1}{2} \Bigl(\langle A \nabla u, \nabla u\rangle - \langle A \nabla h, \nabla h \rangle \Bigr) dx = \int_{B_R} \dfrac{1}{2}  \langle A \nabla (u-h), \nabla(u-h) \rangle \ dx,
\end{equation}
and by ellipticity, 
\begin{equation} \label{TR3}
\int_{B_R} \dfrac{1}{2}  \langle A \nabla (u-h), \nabla(u-h) \rangle  dx	 \ge \dfrac{\lambda}{2} \int_{B_R} |\nabla u - \nabla h|^{2} dx
\end{equation}
Invoking Lemma \ref{DL} we find 
\begin{equation} \label{TR4}
	\begin{array}{lll}
		\displaystyle \int_{B_r} | \nabla u(x) - (\nabla u)_{r}|^2dx &\le& \displaystyle   C \Big( \dfrac{r}{R}\Big)^{n-2 +2\alpha}  \int_{B_R} |\nabla u(x) - (\nabla u)_R|^2 dx \\
		&+& \displaystyle  C\int_{B_R} | \nabla u(x) - \nabla h(x)|^2dx.
	\end{array}
\end{equation}
On the other hand, by minimality of $u$,
 we have,
\begin{eqnarray} \label{TR5}
\int_{B_R} \dfrac{1}{2}(\langle A \nabla u, \nabla u\rangle - \langle A \nabla h, \nabla h\rangle ) dx &\le &   \int_{B_R} (F_0(h) - F_0(u))dx  \nonumber \\
& + & \int_{B_R} f(h-u) dx.
\end{eqnarray}
Clearly,
\begin{equation} \label{TR6}
 \int_{B_R} (F_0(h) - F_0(u))dx \le C(\lambda^{+},\lambda^{-},\|\lambda\|_{L^\infty(\overline{\Omega})})|B_R|.
\end{equation}
If $f\in L^q(\Omega),q > n/2,$ applying H\"older inequality, Poincar\'e inequality and Young's inequality respectively, we  obtain
\begin{eqnarray} \label{TR7}
\int_{B_R} f(h-u) dx & \le & \| f\|_{L^{\frac{2^{*}}{2^{*} -1 }}(B_R)} \|h-u\|_{L^{2^{*}}(B_R)} \nonumber \\
&\le& C \|f\|_{L^{q}(B_R)} |B_R|^{(1/2 -1/q +1/n)} \|\nabla(h-u)\|_{L^{2}(B_R)} \nonumber \\ 
&\le & \dfrac{1}{4} \|\nabla(h - u)\|_{L^{2}(B_R)}^{2} + C'\|f\|_{L^{q}(B_R)}^{2}|B_R|^{1+2(q-n)/nq}   
\end{eqnarray}
Thus, combining (\ref{TR4}), (\ref{TR5}), (\ref{TR6}) and (\ref{TR7}) we reach
\begin{eqnarray*}
\int_{B_r}|\nabla u(x) - (\nabla u)_{r}|^{2} dx & \le & C(n,\lambda,\Lambda) \Bigl( \dfrac{r}{R}\Bigr)^{n-2+2\alpha} \int_{B_R}|\nabla u(x) - (\nabla u)_{R}|^{2} dx \\
&+& C(n,\lambda,\Lambda)C(\lambda_{+}, \lambda_{-},\|\lambda\|_{L\infty(\overline{\Omega})})|B_R| \\
&+& C(n,\lambda,\Lambda) \|f\|_{L^{q}(\Omega)}|B_R|^{1+2(q-n)/nq}\\
&\le & C \left ( \dfrac{r}{R} \right )^{n-2+2\alpha} \int_{B_R}|\nabla u(x) - (\nabla u)_{R}|^{2} dx \\
&+& CR^{n-2 +2(2-n/q)} + CR^n.
\end{eqnarray*}
Finally, employing Lemma \ref{elem.ineq}. together with classical Morrey Theorem, we conclude the proof.
\end{proof}

\section{Euler-Lagrange Equation} \label{sct FBC}
In this section, we comment on the Euler-Lagrange equation associated to the functional $\mathcal{F}$. As previously mentioned, such an equation carries out a flux balance through one phase to another, which represents a singular measure along the transmission boundary. 

Initially, notice that it follows from Theorem \ref{reg} that the positive and negative phases of a minimum are open sets. Thus, fixed a point $x_0$, say in the positive set of a minimum, small perturbations around such a point do not leave that phase. Thus, by classical arguments, 
$$
	\mbox{div}( A_+ \nabla u) =  f_{+}, \quad \text{ in } \{ u > 0 \}.
$$
The same reasoning can be employed within the interior of the negative phase, and we state the conclusion in the following Proposition.  
\begin{proposition}  \label{propweakequation}
Let $u$ be a minimum of  problem \eqref{F1}. Then $u$ satisfies
\begin{equation}
\mbox{div}( A_+ \nabla u) =  f_{+} \quad \mbox{in} \quad \Omega \backslash \{u \le 0\},
\end{equation}
\begin{equation}
\mbox{div}( A_- \nabla u) =  f_{-} \quad \mbox{in} \quad \Omega \backslash \{u \ge 	0\},
\end{equation}   
in the  distributional sense. 
\end{proposition}

Let us now look at the equation satisfied through the free transmission boundary of the minimization problem \eqref{F1}
\begin{equation}\label{free boundary}
	\Gamma(u) := (\partial \{u>0\}\cup \partial \{u<0\}) \cap \Omega.
\end{equation}   
Given a point $x_0 \in \Gamma(u)$, a small ball $B$ centered at $x_0$, a vector field $\Phi \in C^{1}_{0}(B, \mathbb{R}^{n})$ and a number $\epsilon \sim 0$, we define the numbers $\Phi_\epsilon^{+}$ and $\Phi_\epsilon^{-}$ by 
$$
	\Phi^{\pm}_\epsilon := \int \limits_{B \cap \{u=\varepsilon\}}  \langle ( \langle  A_\pm \nabla u,\nabla u \rangle -2 \lambda^{\pm}(x)) \nu, \Phi  \rangle d\H,
$$
where $\nu$ denotes the outward normal vector on  $B \cap \{u=\varepsilon \}$.

\begin{proposition} \label{FBC}  Let $x_0$ be a free boundary point and assume near $x_0$ we have $\Leb(\{u=0\}) =0$. Then for any  $\Phi \in C^{1}_{0}(B, \mathbb{R}^{n})$, there holds
	\begin{eqnarray*}
 	\lim_{\varepsilon_1 \searrow 0} \Phi^{+}_{\epsilon_1} -   \lim_{\varepsilon_2  \nearrow 0} \Phi^{-}_{\epsilon_2} = 0,
	\end{eqnarray*}
	where $ \Phi^{+}_{\epsilon_1}$ and  $\Phi^{-}_{\epsilon_2}$ are defined above.
\end{proposition}

\medskip 

The proof of Proposition \ref{FBC} follows by Hadamard's type of methods, see for instance \cite{T0, MT, KT, LT}, see also \cite{EM} for further domain variation results. We have chosen to omit the details here. It is important to notice that, in particular,  if $A_{+}$ and $A_{-}$ are (separately) H\"older continuous and the free boundary is a $C^{1,\alpha}$ surface, then the flux balance
\begin{equation}
\langle A_+ \nabla u^{+}, \nabla u^{+} \rangle -  \langle A_- \nabla u^{-}, \nabla u^{-} \rangle =  2(\lambda^{+}(x) - \lambda^{-}(x)),
\end{equation}
holds in the classical sense along $\Gamma(u)$.

\section{Geometric nondegeneracy} \label{GN}

In this current Section we show that a local minimum of the energy functional \eqref{F} do grow {\it at least} linearly away from the free boundary. An important tool we shall use here is the non-homogeneous Moser Harnack inequality,  which we state here for completeness.

\begin{theorem}[Moser's Harnack inequality] \label{DHM}
Let $\Omega$ be a domain in $\mathbb{R}^{n}$ and $a_{ij} \in L^{\infty}(\Omega)$ such that
$$ 
	\lambda |\xi|^{2} \le a_{ij} \xi_i \xi_j \le \Lambda |\xi|^{2} \quad \mbox{ for all } x \in \Omega,   \mbox{ and all }  \xi \in \mathbb{R}^{n} 
$$
where $\lambda, \Lambda$ are positive constants. Let $u \in H^{1}(\Omega)$  be a non-negative  weak solution of
\begin{equation}
	\div (a_{ij} \nabla u) = f, \quad \text{ in } \Omega,
\end{equation}
with $f \in L^{q}(\Omega)$ for some $q>n/2$. Then, for any $B_r(x_0) \subset \Omega$ we have
\begin{equation}
\max_{B_r(x_0)}u  \le C \left ( \min_{B_{r/2}(x_0)} u + r^{(2 -n/q)} \|f\|_{L^{q}(B_r(x_0))} \right ),
\end{equation}
where $C=C(n,\lambda,\Lambda,q).$
\end{theorem} 

Hereafter, let us denote
$$
	F^+(u) :=  \partial \{ u>0\} \cap \Omega.
$$
Our next Theorem shows that if the source terms $f_{\pm}$ belongs to $L^{q}$ with $q >n$, then $u^+$ grows linearly inside $\{u>0\}$ out from  $F^{+}(u)$.

\begin{theorem} \label{CL}
Let $u$ be a local minimum of the problem \eqref{F1} with $f \in L^{q}(\Omega)$,  $q >n$. Given $\Omega' \Subset \Omega$, there exists a constant $c_0 > 0$ depending only on dimension, ellipticity, $\Omega'$ and the data of the problem,  such that
\begin{equation} \label{CL1}
 u(x_0) \ge c_0  \cdot   \mathrm{dist} (x_0,F^{+}(u)),
\end{equation}
for any point $x_0 \in \{u>0\} \cap \Omega'$. 
\end{theorem}
\begin{proof}
It suffices to show such an estimate for points  $x_0 \in \{u>0\} \cap \Omega'$ that are close enough to the free boundary, i.e., satisfying 
$$ 
	0 < \mbox{dist}(x_0,F^{+}) \ll \delta,
$$
where $\delta$ depends on dimension, ellipticity, $\Omega'$ and the data of the problem. For that, let us denote by $d:= \mbox{dist}(x_0,F^{+})$ and define the scaled function
$$ 
	v(x) = \dfrac{1}{d} u(x_0 + dx). 
$$
The goal is to show that $v(0) \ge c$ where $c > 0 $ is universal.  Applying classical change of variables,  we see  that $v$ is a local minimum of
\begin{equation} \label{Fd}
 \mathcal{F}^{d}(\xi): = \int\limits_{B_1} \left \{ \dfrac{1}{2} \langle A_+(y) \nabla \xi, \nabla \xi \rangle + \lambda_{+}(y) \chi_{\{\xi > 0\}} + d f_{+}(y)\xi \right \}dx,
 \end{equation}
for $y=x_0 + dx$. By construction, $v>0$ in $B_1$ and by minimality of $v$ we have  
$$
	\mbox{div}(A_{+}^{*}(x) \nabla v) =  f_+^{*} \quad \mbox{in} \quad B_1,
$$
where $A_{+}^{*}(x) = A_+(x_0 +dx)$ and $f_+^{*} = df_{+}(x_0+dx)$.
Harnack inequality displayed in (\ref{DHM}) yields 
\begin{equation} \label{EH}
v(x)   \le  C \left \{ v(0) + d^{1-n/q}\|f_+\|_{L_{q}(\Omega)} \right \}.
\end{equation}
Now choose a non-negative, smooth radially symmetric, cut-off function $\psi$ satisfying
$$ 
	\psi \equiv 0 \quad \mbox{in} \quad B_{1/8}, \quad 0 \le \psi \le 1   \quad \mbox{e} \quad \psi \equiv 1 \quad \mbox{in} \quad B_1 \backslash B_{1/2}
$$
and consider the test function $g$ in $B_1$ given by 
$$ g(x) := \left\{
\begin{array}{ccc}
\min \{v, C \left \{ v(0) + d^{1-n/q}\|f_+\|_{L_{q}(\Omega)}\right \} \psi & \mbox{in} & B_{1/2} \\
v & \mbox{in} & B_1 \backslash B_{1/2}
\end{array}
\right.
$$
We immediately see that 
$$ 
	B_{1/2} \supset \Pi: = \{ y \in B_{1/2}:  C \left \{ v(0) + d\|f_{+}\|_{L^{q}(\Omega)} \right \} \cdot \psi(x) < v(y)\} \supset B_{1/8}.
$$
From minimality of $v$, we estimate
\begin{eqnarray} \label{EMi}
\Sigma &:=&   \int_{\Pi} \lambda_{+}(x_0+dx)(1 -\chi_{\{g>0\}}) + d \cdot f_{+}(x_0+dx)[v(x) -g(x)]dx \nonumber \\ 
&\le& \int_{\Pi} (\langle A_{+}^{*} \nabla g, \nabla g \rangle - \langle A_{+}^{*} \nabla v, \nabla v \rangle) dx \nonumber\\
&\le& \nonumber  
\Lambda \int_{\Pi} (| \nabla g|^{2} )dx \nonumber \\
&\le&  \Lambda \left [ C \left \{ v(0) + d^{1-n/q}\|f_{+}\|_{L^{q}(\Omega)} \right \} \cdot \|\nabla \psi\|_{L^{2}(\Omega)} \right ]^{2} \nonumber \\
&\le&  C v(0)^{2} + C \Big( d^{1-n/q} \|f_{+} \|_{L^{q}(\Omega)}\Big)^{2}. \nonumber \\
\end{eqnarray}
We now aim towards a lower bound for the LHS of (\ref{EMi}). For that we use the strict positiveness of  $\lambda$, namely the fact that  $\lambda \ge c_\lambda > 0$.  Initially we estimate
\begin{equation} \label{EMi2}
\int_{\Pi} \lambda_{+}\lambda(x_0+dx)(1 -\chi_{\{g>0 \} })dx = \int_{\Pi} \lambda_{+}\lambda(x_0+dx) \chi_{\{g=0 \}}dx \ge \lambda_{+} c_\lambda |B_{1/8}|.
\end{equation}
Now, applying once more H\"older inequality and the fact that $ \Pi \subset B_{1/2} $ we find
$$
0\le v-g\le v \le C\{v(0) + d^{1-n/q}\|f_+\|_{L(\Omega)}\} \ \ \mbox{in} \ \ \Pi.
$$
Thus, we  can  estimate 
\begin{eqnarray} \label{EMi3}
\Sigma_2 &:=&  - \int_{\Pi} d f_{+}(x_0 + dx) [ v(x) - g(x)] dx \nonumber \\ 
& \le & d \| v-g \|_{L^{q/(q-1)}(\Pi)}\| f_{+}(x_0+dx) \|_{L^q(\Pi)} \nonumber \\
& \le & d \| C \left \{ v(0) + d \|f_{+}\|_{L^{q}(\Omega)} \right \} \|_{L^{q/(q-1)}(\Pi)}\| f_{+}(x_0+dx) \|_{L^q(\Pi)} \nonumber \\
&\le & d^{1-n/q} C \left \{ v(0) + d^{1-n/q} \|f_{+}\|_{L^{q}(\Omega)} \right \} \|f_{+}\|_{L^{q}(\Omega)}
\end{eqnarray} 

Combining (\ref{EMi}), (\ref{EMi2}) and (\ref{EMi3}) we find
\begin{equation}
C v(0)^{2} + C d^{1-n/q}v(0)  \ge \lambda_{+} c_\lambda |B_{1/8}| - Cd^{2(1-n/q)} \|f_+\|_{L^q(\Omega)}^2
\end{equation}
Thus, if $0<d<\delta(n,\lambda_{+},\lambda(x), \Lambda, \|f_{+} \|_{L^{q}(\Omega)}) \ll 1$ we obtain.
$$ 
	v(0) \ge c(n,\lambda_{+}, \lambda(x), \Lambda, \|f_{+} \|_{L^{q}(\Omega)}) > 0,
$$
as desired. The proof of the Theorem is concluded.
\end{proof}

In the sequel,  we iterate linear growth established in Theorem \ref{CL} as we obtain a stronger non-degeneracy  property for minima of $\mathcal{F}$ near the free boundary.

\begin{theorem} \label{TNDF}
Let $u$ be a local minimum of the problem \eqref{F1}, with $f \in L^{q}(\Omega)$,  $q > n$. Given $\Omega' \Subset \Omega$, there exist a constant $\underline{c}>0$ depending only on dimension, ellipticity, $\Omega'$ and the data of the problem,  such that
$$ 
	\sup_{B_r(x_0)} u^{+} \ge \underline{c}r,
$$
for $x_0 \in \overline {\{u>0\}} \cap \Omega' $ any $0<r < \mathrm {dist}(\partial \Omega', \partial \Omega)$.
\end{theorem}

\begin{proof}
By continuity, Theorem \ref{reg}, it suffices to show the thesis of the Theorem within the positive phase
$$ 
	\Omega_{0}^{+} : = \{u>0\} \cap \Omega'. 
$$
Initially, we start off by showing the existence of a number $\delta_0 > 0$ depending only on dimension, ellipticity, $\Omega'$ and the data of the problem, such that if $x \in \{u > 0\} \cap \Omega'$, then there holds 
\begin{equation} \label{TNDFEq01}
\sup_{B_{d(x)}} u \ge (1 + \delta_0)u(x)
\end{equation}
where $d(x):=\mbox{dist}(x, F^{+})$. In order to verify (\ref{TNDFEq01}) we assume, for the purpose of contradiction, that no such a $\delta_0$ exist. If so, it would be possible to find sequences 
$$
	\delta_j = \text{o}(1) \quad \text{ and } \quad  x_j \in \{u>0\} \cap \Omega'
$$
satisfying 
\begin{equation} \label{AF}
\sup_{B_{d_j}(x_j)} u \le (1 + \delta_j)u(x_j)  
\end{equation}
for $d_j : = \mbox{dist}(x_j,F^{+})=  \text{o}(1)$  as $ j \rightarrow \infty.$ Define the normalized sequence of functions $\rho_j \colon B_1\rightarrow \mathbb{R}$ given by 
$$ \rho_j(z):= \dfrac{u(x_j + d_jz)}{u(x_j)}. $$
Clearly $\rho_j(0)=1$ and from (\ref{AF}), we have
\begin{equation}
0 \le \rho_j \le 1 + \delta_j \quad \mbox{in} \quad B_1.
\end{equation}
Notice that $\rho_j$ satisfies
\begin{equation}
\mbox{div}(A^{*} \nabla \rho_j) = \dfrac{d_{j	}^{2}}{u(x_j)} f_{+}(x_j + d_jz),
\end{equation}
in the distributional sense in $B_1$ where $A^{*}(z)=A_+(x_j + d_jz)$. Thus taking into account the linear growth established in Theorem  \ref{CL}, we obtain
\begin{equation} \label{NDF2}
 |\mbox{div}(A^{*} \nabla \rho_j)| \le C d_j f_{+}(x_j + d_jz) \quad \mbox{in} \quad B_1. 
\end{equation}
By elliptic regularity estimates, we deduce the sequence $\{\rho_j\}$ is equicontinuous in $B_1$. Thus, up to a subsequence $\rho_j \rightarrow \rho$ locally uniformly in $B_1$. Again, by Harnack inequality, for any $x$ such that $|x| \le r < 1$, there holds
\begin{equation}
0 \le [1+ \delta_j] - \rho_j(x) \le C_r \Bigl( [1+\delta_j] - \rho_j(0) -d_j^{1-n/q}\|f\|_{L^{q}(\Omega)} \Bigr) = C_r\cdot  \text{o}(1).
\end{equation}
Letting $j \rightarrow \infty$ in the above estimate, we conclude the limiting blow-up  function $ \rho$ is identically $1$ in $B_1$. On the other hand, let $y_j \in F^{+}$ such that $d_j=|x_j-y_j|$. Up to a subsequence, there would hold
$$ 
	1+ \text{o}(1) = \rho_j \left (\dfrac{y_j-x_j}{d_j} \right )=0,
$$
which clearly gives a contradiction for $j \gg 1$. We have shown the estimate (\ref{TNDFEq01}) does hold true. The conclusion of the proof of Theorem \ref{TNDF} now follows by Caffarelli's polygonal type of argument. That is, we construct a polygonal along which $u$ grows linearly. Starting from $x_0$ and finding a sequence of points $\{x_n\}_{n}$ such that:
\begin{enumerate}
\item $u(x_n) \ge (1 + \delta_0)^{n}u(x_0)$
\item $|x_n - x_{n-1}| = \mbox{dist}(x_{n-1},F^{+})$
\item$ u(x_n) - u(x_{n-1}) \ge  c|x_n-x_{n-1}|$. In	 particular, $|u(x_n) - u(x_0)| \ge c|x_n-x_0|.$
\end{enumerate}

Since $u(x_n)\rightarrow \infty$ as $ n \rightarrow \infty $, there exist a last $x_{n_0}$ in the ball $B_r(x_0)$. Then, for such a point, we have:
$$ 
	|x_{n_0} -x_0| \ge \underline{c}  \cdot r.
$$
Thus,
$$ 
	\sup_{B_r(x_0)} u  \ge u(x_{n_0})\ge u(x_0) + c|x_{n_0} -x_0| \ge  \underline{c}  \cdot r.,
$$
and the proof is concluded.
\end{proof}


\section{Regularity in continuous media} \label{FHR}

In this Section we consider the case where $A_{+} = A_{-} = A(x)$ is  merely a continuous function. As warned in the Introduction, under mere continuity assumption on the media, even solutions to the homogeneous equation $\div (A(x) \nabla u) = 0$ may fail to have bounded gradients. In view of such an obstruction, our ultimate goal in this Section is to show that minima of $\mathcal{F}_A$ are locally $C^{0,1^{-}}$, which is an optimal asymptotic  estimate provided the media is just continuous. 

Our strategy to prove such an estimate is to interpret the Alt-Caffarelli-Friedman theory developed in \cite{ACF} as the {\it tangential} free boundary problem for small coefficient perturbations and sources, see \cite{Caff01} and also \cite{T1, TU, T2, T3, PT} for further applications of geometric tangential methods. We shall establish the following more general result:

\begin{theorem} \label{AC} 
Assume $f_{\pm} \in L^n(\Omega)$. Given an $ \alpha \in (0,1)$, there exists an $ \varepsilon > 0$  depending on $\alpha$ and the data of the problem such that if
$\| A(x) - A_0\|_{L^2(\Omega)} \le \varepsilon $, where $A_0$ is a constant matrix, then minimizers of the functional $\mathcal{F}_A$ are locally  $C^{\alpha}$. Furthermore, for any $\Omega' \Subset \Omega$, there exists a constant $C$ depending on $\Omega'$, $\alpha$ and the data of the problem, such that
$$
	\|u\|_{C^{\alpha}(\Omega')} \le C.
$$
\end{theorem}

The proof of Theorem \ref{AC} will be delivered is the sequel. It shall be divided in three main steps. Initially we establish a powerful approximating Lemma that 
says that if the coefficients do not oscillate much and the source function has small norm, then the graph of a minimum of $\mathcal{F}_A$ is close to a graph of a Lipschitz function, with uniformly bounded norm. Next we show a discrete $C^{0,\alpha}$ estimate at free boundary points which ultimately will lead to the aimed regularity along the free boundary. The final step is to prove, vie geometric considerations, that $C^{0,1^{-}}$ smoothness holds near the free boundary -- not only at free boundary points. 

\begin{lemma} \label{Ind}
For each $\varepsilon >0$ given, there exists a $\delta > 0 $ and a universal constant $C_0 >0$  such that if $u$ is a minimum of $\mathcal{F}_A$ in $B_1$ where 
\begin{equation}
\|A_k - Id\|_{L^2(B_1)} +  \|f_+\|_{L^n(B_1)} + \|f_-\|_{L^n(B_1)} \le \delta,
\end{equation}
and the origin is a free boundary point,  then we can find a Lipschitz function $h$ in $B_{1/2}$, with $\|h\|_{\mathrm{Lip}}(B_{1/2}) \le C_0$, $h(0)=0$ and 
\begin{equation} \label{Apro}
 \| u-h\|_{L^\infty(B_{1/2})}  \le \varepsilon.
\end{equation}
\end{lemma}
\begin{proof}
Let us suppose, for the purpose of contradiction, that the thesis of the Lemma fails. If so, there would exist a positive number $\varepsilon_0 >0$ and sequences $u^k$, $f_-^k$ ,$f_+^k$, and  $A_k $, where $u^k$ is a minimum of the corresponding functional  $\mathcal{F}_{A_k, f^k_{\pm}}$, with
\begin{equation}
\|A_k - Id\|_{L^2(B_1)} + \|f_+^k\|_{L^n(B_1)} +  \|f_-^k\|_{L^n(B_1)} =  \text{o}(1) \quad \mbox{as} \quad k\rightarrow 0;
\end{equation}
however, for any Lipschitz function $h$, satisfying $\|h\|_{\mathrm{Lip}}(B_{1/2}) \le C_0$   and $h(0)=0$ for a constant  $C_0$ to be fixed {\it a posteriori}, we verify
\begin{equation} \label{contract}
 \| u-h \|_{L^\infty(B_{1/2})} > \varepsilon_0.
\end{equation}

From Theorem \ref{LIB}, Remark \ref{rmk H1} and Theorem \ref{reg},  we obtain a uniform control on
$$
	 \|u^k\|_{H^1(\overline{B_{1/2}})} + \|u^k\|_{L^\infty(\overline{B_{1/2}})} + [ u^k  ]_{C^{0,\alpha_0}(\overline{B_{1/2}})} \le K_0.
$$
Thus, up to a subsequence, $u^k$ converges locally uniformly and weakly in $H^1(\overline{B_{1/2}})$  to some function $u^\star$.  In the sequel, we show that the limiting function is a
minimum to the Alt-Caffarelli-Friedman functional $\mathcal{F}_{Id,\lambda_\pm,0}$, studied in \cite{ACF}.  For that we compute
\begin{eqnarray*}
& &\liminf_{k \rightarrow \infty} \hspace{-0.3cm}
\int\limits_{B_{1/2}\cap\{u^k>0\}} \hspace{-0.3cm} \{ |\nabla u^k|^2_{A^k} + \lambda_{+} + f_{+}^{k}u^{k} \}\,dx
+
\hspace{-0.3cm} \int\limits_{B_{1/2}\cap\{u^k \le 0\}} \hspace{-0.3cm} \{ |\nabla u^k|^2_{A_k}+ \lambda_{-} + f_{-}^{k}u^{k} \}\,dx \\
&=& \liminf_{k \rightarrow \infty} 
\int\limits_{B_{1/2}\cap\{u^k>0\}}\hspace{-0.3cm} \{ |\nabla u^k|^2 + \lambda_{+} + f_{+}^{k}u^{k} \}\,dx
+
\hspace{-0.3cm} \int\limits_{B_{1/2}\cap\{u^k \le 0\}}  \hspace{-0.3cm} \{ |\nabla u^k|^2+ \lambda_{-} + f_{-}^{k}u^{k} \}\,dx \\
& + &\liminf_{k\rightarrow \infty} \hspace{-0.3cm}
\int\limits_{B_{1/2}\cap\{u^k>0\}} \hspace{-0.3cm}
 \langle (A^k - Id)\nabla u^k,
\nabla u^k\rangle \,dx + \hspace{-0.3cm}
\int\limits_{B_{1/2}\cap\{u^k \le 0\}} \hspace{-0.3cm}
 \langle (A^k- Id)\nabla u^k,
\nabla u^k\rangle\,dx
\\ &\ge &
\int\limits_{B_{1/2}\cap\{u^\star>0\}}\{ |\nabla u^\star|^2 + \lambda_{+}  \} dx
 + \int\limits_{B_{1/2}\cap\{u^\star \le 0\}} \{ |\nabla u^\star|^2 + \lambda_{-}   \} dx.
  \\ 
 & = & \int\limits_{B_{1/2}} [| \nabla u^\star|^2  + \lambda^2(u^\star)]dx
\end{eqnarray*}
where $\lambda^2$ is like in \cite{ACF}. Thus, indeed, $u^\star$ is a minimum of  the functional studied in \cite{ACF}   and by the uniform convergence of $u^k$ we get $u^\star(0) = 0$.  As a consequence of monotonicity formula, $u^{\star}$ is locally Lipschitz continuous, and for a dimensional constant $C_n$, there holds
$$
	\|u^\star\|_{\mathrm{Lip}}(B_{1/2}) \le C_n  \cdot \|u^\star\|_{L^2(B_{2/3})}.
$$
Clearly, from the convergence above,
$$
	 \|u^\star\|_{L^2(B_{2/3})} \le 2 \|u^k\|_{L^2(B_{2/3})} \le C_1,
$$
where $C_1$ depends only of data of the problem. Set $C_0 = C_n \cdot C_1$. If we take $h = u^\star$ and $k \gg 1$, we reach a contradiction in \eqref{contract} and the Lemma is proven. 
\end{proof}

In the sequel we show a discrete version of the aimed H\"older regularity estimate.

\begin{lemma}
Given $0 < \alpha < 1$, there exist $0< \delta_0 <1$ and $0<\rho<1/2$  depending only upon $\alpha$ and the data of the problem, such that if $u$ is a minimum of $\mathcal{F}_A$ with  $\|A - Id\|_{L^2(B_1)} \le \delta_0,  \|f_+\|_{L^n(B_1)}, \|f_-\|_{L^n(B_1)} \le \delta_0 $. Then 
\begin{equation}\label{Step}
\sup_{x \in B_{\rho^k}} |u(x)|\le \rho^{k\alpha}.
\end{equation}
\end{lemma}

\begin{proof} 
Let us first argue for the case $k=1$.  For $\varepsilon$ to be chosen, let $h$ be the Lipschitz function granted by Lemma \ref{Ind} that is within an $\varepsilon$ distance to $u$ in the $L^\infty$ topology. Define $c_1:=h(0)$. We can estimate
\begin{equation}\label{pp1}
\sup_{x\in B_{\rho}} |u(x)-c_1| \le
\sup_{x\in B_{\rho}} |u(x)-h(x)|+\sup_{x\in B_{\rho}}|h(x)-h(0)|\le
\varepsilon +C_0\rho.
\end{equation}
Now we select
\begin{equation} \label{def rho}	
	\rho :=  (2C_0)^{-\frac{1}{{(1 - \alpha)}}}, \quad 	 \varepsilon := \dfrac{1}{2} \rho^{\alpha},
\end{equation}
and aimed estimate for $k=1$ is proven. We now procedure by induction. Suppose we have verified  \eqref{Step} for  $j = 1, 2, \cdots, k$; we must show the estimate is valid for $ j = k+1$. To this end, define
\begin{equation} \label{scaling}
 \tilde u(x):=\frac{u(\rho^k x)}{\rho^{k\alpha}}. 
\end{equation}
One verifies that $\tilde{u}$ is a minimizer of 
\begin{equation} \label{Ftil}
	\begin{array}{lll}
	 	{\mathcal{\tilde F}}_{\tilde A,, \tilde \lambda_{\pm}, \tilde f_{\pm} }(v)&:= & \displaystyle  \int\limits_{B_1 \cap\{ v> 0\}} \left \{\dfrac{1}{2} |\nabla v(x)|^2_{\tilde A} + \tilde \lambda_{+}(x) + \tilde f_{+}v \right \}  \,dx \\
		&+&
		\displaystyle  \int\limits_{B_1 \cap\{ v\le 0\}} \left \{\dfrac{1}{2}|\nabla v(x)|^2_{\tilde A_-}  +\tilde  \lambda_{-}(x) + \tilde f_{-}v \right \}\,dx,
	\end{array}
\end{equation}	
over $H^1_{\tilde u}(B_1) $, where 
\begin{equation}
\tilde A(x) := A(\rho^k x), \ \tilde \lambda_{\pm}(x) := \rho^{2k(1-\alpha)} \lambda_{\pm}(x) \ \mbox{and} \ \tilde f_{\pm}(x) : = \rho^{k(2-\alpha)}f_{\pm}(x).
\end{equation}
In fact, by change of variables we can see that
\begin{equation}
\begin{array}{lll}
	 	{\mathcal{\tilde F}}_{\tilde A,\tilde \lambda_{\pm}, \tilde f_{\pm}}(\tilde{u})&:= & \dfrac{\rho^{2k(1-\alpha)}}{\rho^{kn}}\Biggl( \ \displaystyle  \int\limits_{B_{\rho^k} \cap\{ v> 0\}} \left \{\dfrac{1}{2} |\nabla u(x)|^2_{ A} +  \lambda_{+}(x) +  f_{+}u \right \}  \,dx \\
		&+&
		\displaystyle  \int\limits_{B_{\rho^k} \cap\{ v\le 0\}} \left \{\dfrac{1}{2}|\nabla u(x)|^2_{ A}  +  \lambda_{-}(x) +  f_{-}u \right \}\,dx\Biggr).
	\end{array}
\end{equation}	
And the assertion follows by minimality of $u$.
Notice that $\tilde u$ is a minimizer in $B_1$ with respect to the elliptic
matrices $\tilde A(x):=A(\rho^k x)$, where
$$ \|\tilde A-Id\|_{L^\infty(B_1)} \le \|A - Id\|_{L^\infty (B_{\rho^k})}\le\epsilon.$$ 
and 
$$ \|\tilde f_\pm \|_{L^n(B_1)}^n = \rho^{k(2-\alpha)n}\int_{B_1} |f_\pm(\rho^kx)|^n dx = \rho^{kn(1-\alpha)}\|f_\pm\|_{L^n(B_{\rho^k})} \le \varepsilon^n.$$
Therefore, applying the case $k=1$ already proved, we obtain
$$ \sup_{x \in B_\rho} | \tilde{u}(x)| \le \rho^\alpha.$$
Equivalently,
$$ 
	\sup_{x \in B_{\rho^{k+1}}}|u(x)| \le \rho^{(k+1)\alpha}. 
$$

\end{proof}

The proof of  Theorem \ref{AC}  at free boundary points now follows.  Indeed fix $x$ in $B_\rho$, where $\rho$ is the universal number declared in \eqref{def rho}. Choose $k$ in such a way that
$$ 
	\rho^{k+1}< |x|\le \rho^k.
$$
Using \eqref{Step}, we conclude  
$$
	 |u(x)| \le \left ( \rho^{-\alpha}\right )  \cdot |x|^\alpha,
$$
as desired. \hfill \qed

\bigskip

At this stage we have obtained $C^{0,1^-}_{\loc}$ regularity  at points along the free boundary. Since $A$ is continuous, $f$ lies in $L^n$, from the Euler-Lagrange equation satisfied in each phase, Proposition \ref{propweakequation}, we also have local $C^{0,1^-}$  regularity estimates away from the free boundary -- see for instance \cite[Theorem 4.2]{T}.  However such a local H\"older estimate, in principle, deteriorates as one approaches the free boundary. In fact, if $x_0$ is a generic point in the set of positivity of $u$, say, and $d := \dist(x_0, \partial \{u>0\})$, then Shauder type estimate gives
\begin{equation} \label{shauderestimate}
\|u\|_{C^{0,\alpha}(B_{d/2}(x_0))} \le \dfrac{C(n,\lambda, \Lambda, \|f_+\|_{L^n} ,\alpha)}{d^\alpha}\|u\|_{L^\infty(B_{2d/3}(x_0))}.
\end{equation}

In the sequel, we will show a universal geometric argument that overcome such a difficulty, even though no smoothness information on the free boundary is prior known. 

The argument starts by applying Harnack inequality on \eqref{shauderestimate}, as to refined the Shauder H\"older estimate \eqref{shauderestimate} to
\begin{equation} \label{SE and Harnack}
\|u\|_{C^{0,\alpha}(B_{d/2}(x_0))} \le \dfrac{C_1}{d^\alpha} \left \{ u(x_0) + d\|f\|_{L^n(B_{2d/3}(x_0))}\right \}.
\end{equation}
Now, let $y_0$ be a free boundary point that realizes the distance, i.e.,
$$
	d := \dist(x_0, \partial \{u>0\}) = |x_0 - y_0| \ll \dist (x_0, \partial \Omega).
$$

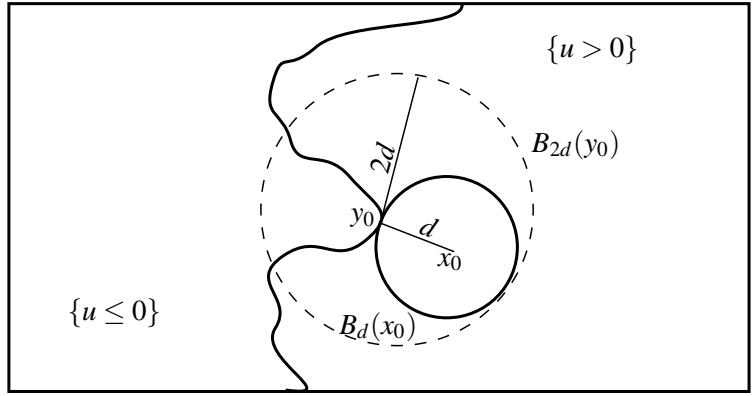
\begin{figure}[!h]
\centering
\scalebox{1} 
{

\begin{pspicture}(0,-2.6590133)(9.88,2.6390135)
\psbezier[linewidth=0.04](6.044045,2.6190133)(6.1,2.2713013)(4.5277443,2.4103038)(4.1050467,2.0426435)(3.6823494,1.6749834)(3.662391,2.0418715)(3.5260675,1.56564)(3.389744,1.0894084)(3.6566293,1.2082454)(3.7656898,0.7959683)(3.8747504,0.3836912)(4.2092223,0.7101972)(4.546211,0.3210444)(4.8831997,-0.06810843)(5.200924,-0.16271722)(4.7667933,-0.52703404)(4.332662,-0.89135087)(4.1625667,-0.39600283)(3.6638827,-0.73905367)(3.1651986,-1.0821044)(3.8222706,-1.293778)(3.5620756,-1.6973823)(3.3018804,-2.1009865)(4.502354,-2.6390135)(3.7,-2.5209866)
\psframe[linewidth=0.04,dimen=outer](9.88,2.6390135)(0.0,-2.5609865)
\usefont{T1}{ptm}{m}{n}
\rput(1.421455,-1.4959866){$\{u\le0\}$}
\usefont{T1}{ptm}{m}{n}
\rput(7.7714553,2.0040135){$\{u > 0 \}$}
\pscircle[linewidth=0.04,dimen=outer](5.84,-0.6209866){0.96}
\pscircle[linewidth=0.02,linestyle=dashed,dash=0.16cm 0.16cm,dimen=outer](5.18,-0.12098658){1.82}
\usefont{T1}{ptm}{m}{n}
\rput(5.891455,-0.8159866){$x_0$}
\usefont{T1}{ptm}{m}{n}
\rput(4.731455,-0.23598658){$y_0$}
\psline[linewidth=0.02cm](5.94,-0.6809866)(4.96,-0.3009866)
\usefont{T1}{ptm}{m}{n}
\rput{-20.044744}(0.45948336,1.9089264){\rput(5.611455,-0.35598657){$d$}}
\psline[linewidth=0.02](4.96,-0.3009866)(5.46,1.6390134)(5.46,1.6390134)
\usefont{T1}{ptm}{m}{n}
\rput{74.75471}(4.244404,-4.3669963){\rput(4.961455,0.5840134){$2d$}}
\usefont{T1}{ptm}{m}{n}
\rput(4.931455,-1.6759865){$B_d(x_0)$}
\usefont{T1}{ptm}{m}{n}
\rput(7.5614552,0.7440134){$B_{2d}(y_0)$}
\end{pspicture} 
}

\caption{$C^{0,\alpha}$ regularity near the free boundary.}
\end{figure}

Since we have already established a universal control on the  $\alpha$-H\"older norm of $u$ along $\partial \{u>0\}$, we can assure the existence of a constant $C_2>0$ such that
$$
	\sup_{B_{2d}(y_0)} |u| \le C_2 (2d)^\alpha.
$$
In particular, as $x_0 \in B_{2d}(y_0)$, we have
\begin{equation} \label{pointwise est holder}
	u(x_0) \le \left ( 2^\alpha C_2 \right ) \cdot d^\alpha.
\end{equation}
Combining \eqref{pointwise est holder} and \eqref{SE and Harnack} we obtain 
$$
\|u\|_{C^{0,\alpha}(B_{d/2}(x_0))} \le C_1(2^\alpha C_2 + \mbox{diam}(\Omega)^{1-\alpha}\|f\|_{L^n(\Omega)}),
$$
and the local $C^{0,1^{-}}$ regularity of $u$ in $\Omega$, thesis of Theorem   \ref{AC}, is finally concluded.  \hfill \qed

\bigskip

As previously anticipated, Theorem \ref{AC} implies that minimum of $\mathcal{F}_A$, with $A$ continuous is locally $C^{0,1^{-}}$. 

\begin{corollary} \label{C1Loc}
If $A \in C^0(\Omega) $ then minimizers of the problem \eqref{F1} where $A_+(x)=A_-(x)=A(x)$ and $f_+,f_- \in L^n(\Omega)$ are $C^{0,1^-}_{\loc}(\Omega)$.
\end{corollary}

\begin{proof}
Let $\Omega' \Subset  \Omega$, set $d := \mbox{dist}(\Omega, \Omega') $ and $\Omega'':= \{ x \in \Omega \suchthat \dist(x, \Omega') \le d/2\} $. Since $A \in C^0(\Omega)$, it is uniformly continuous in $\overline{\Omega''}$. Now, given $0< \alpha < 1$, we choose the corresponding $\varepsilon >0$ from Theorem \ref{AC} and take $0< \delta< d/2$ such that $|A(x) -A(y)|< \varepsilon $ whenever  $x,y \in \overline{\Omega''}$ and  $|x-y|<\delta$. Fix your favorite $x_0 \in \Omega''$ and declare $A=: A(x_0)$. Thus  $\|A(x) - A\|_{L^\infty(B_\delta(x_0))}< \varepsilon $ and by Theorem (\ref{AC}), properly scaled, gives that $u \in C^\alpha(B_\delta(x_0))$.
\end{proof}

\section{Small jumps and asymptotic Lipschitz estimates} \label{sct small jumps}

In this final Section, we present an asymptotic regularity estimate that states that given any $\alpha \in (0,1)$ a minimum of the free transmission problem is of class $C^{0,\alpha}_{\loc}$ provided the heterogeneous media $A_{+}$ and $A_{-}$ are sufficiently close in the $L^2$ norm. In such a perspective, one can see Corollary \ref{C1Loc}, as the particular case when the media have null distance. The strategy of showing such a result goes along the lines of the previous section; however at each step of the iterative process, approximating functions shall be a minimum of a problem contemplated by Corollary \ref{C1Loc}.  

Thus, within this Section, we work under the set-up where $A_{+}$ and $A_{-}$ are (separately) continuous and we fix a given modulus of continuity for $A_{+}$ and $A_{-}$, namely
\begin{equation}\label{mod cont}
	|A_{\pm}(x) - A_{\pm}(y)| \le \omega(|x-y|).
\end{equation}
As previously highlighted, even if $A_{+}$ and $A_{-}$ are (different) constant matrices, one cannot establish monotonicity formula for the homogeneous $\mathcal{F}_{A_{+}, A_{-},0,0}$ functional and Lipschitz estimates may fail. Notwithstanding, we are able to obtain the following asymptotic sharp regularity estimate. 

\begin{theorem}[Improved H\"older regularity] \label{HoR}
Under condition \eqref{mod cont},  given any $\alpha\in(0,1)$ and $\Omega' \Subset \Omega$,
there exists an $\epsilon>0$, depending only upon $n$, $\lambda$,
$\Lambda$, $\omega$, $\Omega$, $\Omega'$ and $\alpha$,  such that if 
\begin{equation} \label{smalljump}
\|A_+ - A_- \|_{L^{2}(\Omega')}\le\epsilon
\end{equation}  
then~$u\in C_\loc^{0,\alpha}(\Omega')$.
\end{theorem}

\begin{proof} Since our estimates are of local character, after proper restriction and scaling, we may assume, with no loss of generality, that $\Omega=B_1$ and the origin is a free boundary point.

We will initially prove that fixed a $\delta>0$, an $ \alpha
\in(0,1)$ and a $0< \tau < 1-\alpha$, there exist $\epsilon$, $C_0>0$,  that depend only upon $n$, $\lambda$,$\Lambda$, $\omega$,  $\tau$ and $\delta$,
and a function $h \in C^{ \alpha + \tau}(B_{1/2})$ such that if
$u\in H^1_{u_0}(B_{1/2})$ is a 
minimizer of $\mathcal{F}_{A_{+}, A_{-}}$ in $B_{1/2}$ with respect to some Dirichlet datum $u_0$, with $L^2$ norm under control, then 
\begin{equation}\label{C1}
 \|A_+-A_-\|_{L^{2}(B_{9/10})}\le\epsilon 
\end{equation}
implies that
\begin{equation}\label{C11}
	\| u-h\|_{L^\infty(B_{2/5})}\le\delta \quad {\mbox{ and }} \quad  \|h\|_{C^{\alpha + \tau}(B_{1/4})}\le C_0.
\end{equation}
 
We verify the above claim by compactness methods. We suppose, searching a contradiction, that such a fact does not hold. This implies the existence of a sequence of $\epsilon\searrow0$
and the existence of families of elliptic matrices $A_+^\epsilon$ and  $A_-^\epsilon $ subjected to the assumptions above, and corresponding minimizer $u^\epsilon$ to $\mathcal{F}_{A_{+}, A_{-}}$, that are within an honest distance from universally bounded $C^{ \alpha + \tau}$ functions. Since the family $ A_+^\varepsilon $ is bounded by ellipticity and equicontinuous by virtue of \eqref{mod cont},  we may assume, up to subsequence, that~$A_+^\epsilon$ converges in $ L^{\infty}_{\loc} (B_1)$ to some $ A_+^0 $ and $ A_-^\varepsilon $ converges in $ L^\infty_{\loc} (B_1) $ to $ A_-^{0} $. By the closeness condition \eqref{C1}, we deduce $A_+^0 = A_-^0 =: A^\star $. 

From Theorem \ref{reg} and Remark \ref{rmk H1}, we know that
$$
	\|u^\epsilon\|_{H^1(\overline{B_{1/2}})} \quad \text{and} \quad \|u^\epsilon\|_{C^{0,\alpha_0}(\overline{B_{1/2}})}
$$	
are bounded uniformly in $\epsilon$. Up to a subsequence $u^\varepsilon$ converges to a function $u^\star$ in $L^\infty(\overline{B_{1/2}})$ and weakly in $H^1(\overline{B_{1/2}})$.
A contradiction  will be obtained once we verify that $u^\star$ is a minimum of a functional $\mathcal{F}_{A^\star}$, for which  Corollary \ref{C1Loc} grants local $C^{0, \alpha+\tau}$ universal estimates. From the uniform convergences $A^\varepsilon_{\pm} \rightarrow A^\star$ and $u^\varepsilon \rightarrow u^\star$ in $L^\infty(B_{1/2})$ and weak convergence $u^\varepsilon \rightharpoonup u^\star$ in $H^1(B_{1/2})$ we compute
\begin{eqnarray*}
& &\hspace{-0,4cm} \liminf_{\epsilon\searrow0} \hspace{-0,4cm}
\int\limits_{B_{1/2}\cap\{u^\epsilon>0\}} \hspace{-0,5cm} \{ |\nabla u^\epsilon|^2_{A^\epsilon_+} +  \lambda_{+} + f_{+}u^{\varepsilon} \}\,dx 
\ + \hspace{-0,6cm}
\int\limits_{B_{1/2}\cap\{u^\epsilon<0\}} \hspace{-0,5cm} \{ |\nabla u^\epsilon|^2_{A^\epsilon_-}+ \lambda_{-} + f_{-}u^{\varepsilon} \}\,dx \\
&=& \hspace{-0,4cm} \liminf_{\epsilon\searrow0} \hspace{-0,5cm}
\int\limits_{B_{1/2}\cap\{u^\epsilon>0\}} \hspace{-0,5cm} \{|\nabla u^\epsilon|^2_{A^\star}+ \lambda_{+} + f_{+}u^{\varepsilon} \} \,dx \ + \hspace{-0,6cm}
\int\limits_{B_{1/2}\cap\{u^\epsilon<0\}} \hspace{-0,4cm} \{|\nabla u^\epsilon|^2_{A^\star}+ \lambda_{-} + f_{-}u^{\varepsilon} \}\,dx
\\ &\qquad\qquad \\
& & + \hspace{-0,3cm} \hspace{-0,2cm}
\int\limits_{B_{1/2}\cap\{u^\epsilon>0\}} \hspace{-0,2cm}
 \langle (A^\epsilon_+ -A^\star)\nabla u^\epsilon,
\nabla u^\epsilon\rangle \,dx \ + \hspace{-0,3cm}
\int\limits_{B_{1/2}\cap\{u^\epsilon<0\}} \hspace{-0,2cm}
 \langle (A^\epsilon_- -A^\star)\nabla u^\epsilon,
\nabla u^\epsilon\rangle\,dx
\\ &\ge & \hspace{-0,6cm}
\int\limits_{B_{1/2}\cap\{u^\star>0\}} \hspace{-0,2cm}\{ |\nabla u^\star|^2_{A^\star} + \lambda_{+} + f_{+}u^{\star} \} dx
\  +  \hspace{-0,4cm} \int\limits_{B_{1/2}\cap\{u^\star<0\}} \hspace{-0,2cm} \{ |\nabla u^\star|^2_{A^\star} + \lambda_{-} + f_{-}u^{\star} \} dx.
 \end{eqnarray*}
This shows that 
$$\liminf_{\varepsilon \searrow 0} \mathcal{F}_{A^\varepsilon_{\pm}}(u^\varepsilon) \ge \mathcal{F}_{A^\star}(u^\star).$$
 Moreover, for any $v\in H^1_{u_0}(B_1)$ we have
\begin{eqnarray*}
\mathcal{F}_{A^\star}(v) & = & \hspace{-0,2cm }\int\limits_{B_{1/2}\cap\{v>0\}} \hspace{-0,2cm} \{ |\nabla v|^2_{A^\star}  +   \lambda_{+} + f_{+}v \} \,dx
\  + \hspace{-0,2cm} \int\limits_{B_{1/2}\cap\{v<0\}} \hspace{-0,2cm} \{ |\nabla v|^2_{A^\star} + \lambda_{-} + f_{-}v \} \,dx\\
&= &\hspace{-0,2cm}  \int\limits_{B_{1/2}\cap\{v>0\}} \hspace{-0,2cm} \{|\nabla v|^2_{A^\varepsilon_+} + \lambda_{+} + f_{+}v \} \,dx \ + \hspace{-0,2cm}
\int\limits_{B_{1/2}\cap\{v<0\}} \hspace{-0,2cm} \{|\nabla v|^2_{A^\varepsilon_-}+ \lambda_{-} + f_{-}v \}\,dx\\ 
\\ &+ & \hspace{-0,2cm}
 \int\limits_{B_{1/2}\cap\{v>0\}}  \hspace{-0,2cm} \langle (A^\star - A^\varepsilon_+) \nabla v, \nabla v \rangle dx
\  + \int\limits_{B_{1/2}\cap\{v<0\}} \hspace{-0,2cm} \langle (A^\star - A^\varepsilon_-) \nabla v, \nabla v \rangle  dx.\\
&=& \mathcal{F}_{A_{\pm}^\varepsilon}(v) + \text{o}(1),
\end{eqnarray*}
as $\varepsilon \searrow 0$. By the minimality of $u^\epsilon$, for any $v\in H^1_{u_0}(B_{1/2})$ there holds
$$ 
	\mathcal{F}_{A^\star}(v) =\lim_{\varepsilon \searrow 0} \mathcal{F}_{A^\varepsilon_\pm}(v) \ge \liminf_{\varepsilon \searrow 0} \mathcal{F}_{A^\varepsilon_{\pm}}(u^\varepsilon) \ge \mathcal{F}_{A^\star}(u^\star). 
$$
Accordingly, $u^\star$ is a  minimizer of the functional (\ref{F}) with  $A_+=A_-=A^\star$. From Corollary \ref{C1Loc}, the limiting function $u^{\star}$ lies in $C^{\alpha + \tau}(B_{1/4}) $ and 
$ \|u^\star\|_{C^{\alpha +\tau}(B_{1/4})}\le C_0$, for some universal $C_0>0$. Thus, taking $h:=u^\star$ and $\epsilon$ is small enough, drives us to a contradiction on our starting
assumption.  

\bigskip

In the sequel, we claim that there exists $\rho \in (0,1/4)$ such that
\begin{equation}\label{step}
\sup_{x\in B_{\rho^k}} |u(x)|\le \rho^{k\alpha},
\end{equation}
for any $k\ge 1$. The proof is by induction. To verify the case~$k=1$, we employ triangular inequality and the existence of universal approximating $C^{\alpha+\tau}$ functions. That is, we estimate
$$ 
	\sup_{x\in B_{\rho}} |u(x)|\le \sup_{x\in B_{\rho}} |u(x)-h(x)|+ \sup_{x\in B_{\rho}}|h(x)-h(0)|\le
	\delta+C_0\rho^{\alpha + \tau} \le \rho^\alpha,
$$
if we choose $ \rho $ small enough such that 
$$
	 \rho < \frac{1}{(2C_0)^{1/\tau}}
$$
and $\delta = \frac{1}{2} \rho^\alpha$. Verified \eqref{step}  for  $k\ge1$
we now prove  $(k+1)$th step. For that, we consider the scaled function
$$
	v(x) := \frac{1}{\rho^\alpha} u(\rho x).
$$
As before, $v$ is a minimum of the fucntional 
\begin{equation} \label{FA+til}
\begin{array}{lll}
	 	{\mathcal{\tilde F}}_{\tilde A_{\pm}, \lambda_{\pm} f_{\pm} }(v)&:= & \displaystyle  \int\limits_{B_1 \cap\{ v> 0\}} \left \{\dfrac{1}{2} |\nabla v(x)|^2_{\tilde A} + \tilde \lambda_{+}(x) + \tilde f_{+}v \right \}  \,dx \\
		&+&
		\displaystyle  \int\limits_{B_1 \cap\{ v\le 0\}} \left \{\dfrac{1}{2}|\nabla v(x)|^2_{\tilde A_-}  +\tilde  \lambda_{-}(x) + \tilde f_{-}v \right \}\,dx,
	\end{array}
\end{equation}
over $H^1_{\tilde u}(B_1) $, where 
\begin{equation}
\tilde A_{\pm}(x) := A_{\pm}(\rho^k x), \ \tilde \lambda_{\pm}(x) := \rho^{2k(1-\alpha)} \lambda_{\pm}(x) \ \mbox{and} \ \tilde f_{\pm}(x) : = \rho^{k(2-\alpha)}f_{\pm}(x).
\end{equation}
Thus we can apply the induction step $k=1$ to $v$ and conclude the induction process. 

We have proven $C^{0,\alpha}$ regularity of $u$  along the free boundary. To obtain locally $C^{0,1^-} $ regularity, we use the same geometric argument  employed at the end of the proof of Theorem \ref{AC}. 
\end{proof}

\end{document}